\documentclass[11pt,a4paper]{article}

\usepackage{amsmath, amssymb, amsthm}
\usepackage[numbers,sort,compress]{natbib}
\usepackage{geometry}
\usepackage{authblk}

\newtheorem{theorem}{Theorem}[section]
\newtheorem{lemma}[theorem]{Lemma}

\newtheorem{definition}[theorem]{Definition}
\newtheorem{remark}[theorem]{Remark}

\newcommand{\rank}{\operatorname{rank}}
\newcommand{\adj}{\operatorname{adj}}
\newcommand{\diag}{\operatorname{diag}}
\newcommand{\col}{\operatorname{col}}

\title{Constructions of almost controllable graphs determined by their generalized spectra}
\author
{
	Wei Wang$^{\rm a}$\thanks{Corresponding author\\
		E-mail addresses:  wangwei.math@gmail.com (W. Wang); 1975822059@qq.com (M. Shi);  fenjinliu@163.com (F. Liu).}\quad
	Manjin Shi$^{\rm a}$\quad
	Fenjin Liu$^{\rm b}$
	\\
	{\footnotesize$^{\rm a}$School of Mathematics, Physics and Finance, Anhui Polytechnic University, Wuhu 241000, P. R. China}\\
	{\footnotesize$^{\rm b}$School of Science, Chang'an University, Xi'an 710064, P.R. China}
}
\date{}

\begin{document}
	
	\maketitle
	
	\begin{abstract}
		Identifying and constructing graphs that are determined by their generalized spectrum (DGS) is a significant and challenging problem in spectral graph theory. Recently, a simple criterion for almost controllable graphs to be DGS was proposed by Lin et al.~(2026), utilizing the modified walk matrix. In this paper, we investigate the evolution of the modified walk matrix under disjoint union and join operations with a singleton vertex. We establish an exact algebraic identity for the determinant of the modified walk matrix of the resulting graph. Based on this identity and the DGS-criterion of Lin et al., we construct infinite families of almost controllable graphs that are DGS, extending the previous construction of Liu et al.~(2019), which was restricted to controllable graphs.
			\end{abstract}
			
		\noindent\textbf{Keywords:} generalized spectrum; determined by generalized spectrum; modified walk matrix; almost controllable graph; main polynomial.\\
		
		\noindent\textbf{Mathematics Subject Classification:} 05C50
	\section{Introduction}
	
	Let $G$ be a simple undirected graph on $n$ vertices with adjacency matrix $A$. The \emph{spectrum} of $G$ is the multiset of eigenvalues of $A$. The \emph{generalized spectrum} of $G$ is defined as the spectrum of $G$ together with the spectrum of its complement $G^c$. A graph $G$ is said to be \emph{determined by its generalized spectrum} (DGS) if any graph $H$ sharing the same generalized spectrum is isomorphic to $G$. The study of generalized spectral characterizations of graphs was initiated by Wang and Xu \cite{WangXu2006} as a natural extension of the classical spectral characterization problem, which aims to resolve the long-standing question of which graphs are uniquely determined by their spectra. For an overview on the development of spectral characterizations of graphs, we refer the readers to the survey papers \cite{Dam2003,Dam2009,Sason2025}.
	
	 A fundamental tool in studying the problem of generalized spectral characterizations is the \emph{walk matrix}, defined as $W(G) = [e, Ae, \dots, A^{n-1}e]$, where $e$ is the all-ones vector. A graph $G$ is called \textit{controllable} \cite{Godsil2012} if $\rank W(G)=n$; and \emph{almost controllable} \cite{Wang2021} if $\rank W(G)=n-1$. It is known that the number $2^{-\lfloor n/2\rfloor}\det W(G)$ is always an integer \cite{Wang2013}. Wang \cite{Wang2017} established the following simple DGS-criterion for controllable graphs.
	
	\begin{theorem}[\cite{Wang2017}]\label{wt}
		Let $G$ be a controllable graph with $n$ vertices. If $2^{-\lfloor n/2\rfloor}\det W(G)$ is odd and square-free, then $G$ is DGS.
	\end{theorem}
	
	Using this criterion, Liu et al. \cite{Liu2019} constructed infinite families of controllable DGS graphs by alternately applying disjoint union ($\cup$) and join ($\vee$) operations with a single vertex. Precisely, they proved the following theorem.
	
	\begin{theorem}[\cite{Liu2019}]\label{lc}
		Let $G_0$ be an $n_0$-vertex controllable graph such that $2^{-\lfloor n_0/2\rfloor} \det W(G_0)$ is odd and square-free. Let $G_1,G_2,\ldots$ be a sequence of graphs defined by 
		\begin{equation*}
			G_i=\begin{cases}
				G_{i-1}\cup K_1,&\text{if $i$ is odd,}\\
				G_{i-1}\vee K_1,&\text{if $i$ is even.}
			\end{cases}
		\end{equation*}
		If $\{|\det A(G_0)|,|\det A(G_1^c)|\}=\{1,2\}$, then each graph $G_i$ satisfies the condition of Theorem \ref{wt} and hence is DGS.
	\end{theorem}
	
	\begin{remark}\normalfont
		Note that any adjacency matrix $A$ of odd order is singular over the binary field $\mathbb{F}_2$ as it is an alternating matrix. This implies that a graph with a unimodular adjacency matrix must have an even order. Thus, the condition $\{|\det A(G_0)|,|\det A(G_1^c)|\}=\{1,2\}$ can be equivalently restated as 
		\begin{equation*}
			(|\det A(G_0)|,|\det A(G_1^c)|)=\begin{cases}
				(1,2), & \text{if $n_0$ is even,} \\
				(2,1), & \text{if $n_0$ is odd.}
			\end{cases}			
		\end{equation*} 
	\end{remark}
	
	Another structural method for constructing DGS families is the rooted product construction. Mao, Liu, and Wang \cite{Mao2015} proved that the rooted product graph $G \circ P_2$ (obtained by attaching a pendant edge at every vertex of $G$) is DGS whenever $\det W(G) = \pm 2^{n/2}$ and $\det A(G)=\pm 1$. This construction was subsequently extended to $G\circ P_m$ (for any $m\ge 2$) and $G\circ H^{(v)}$ for a general graph $H$ rooted at an arbitrary vertex $v$; see \cite{Mao2022,Wang2026,Yan2026,Wang2024}.
	
	We note that all existing constructions of DGS families are only applicable to \emph{controllable} graphs. The main aim of this paper is to extend the method of Liu et al.~\cite{Liu2019} to almost controllable graphs, using the DGS criterion due to Lin et al.~\cite{Lin2026}. To state our main results, we first recall this criterion along with several necessary definitions. For convenience, we let $\mathcal{H}_n$ denote the family of almost controllable graphs of order $n$. For a graph $G\in\mathcal{H}_n$, the \emph{modified walk matrix} $\widetilde{W}(G)$ introduced in \cite{Wang2021}, is defined as 
	\begin{equation*}
		\widetilde{W}(G)=[e,Ae,\ldots,A^{n-2}e,2^{1-\lfloor\frac{n}{2}\rfloor} \xi(G)],
	\end{equation*} 
	where 
	\begin{equation*}
		\xi(G)=(W_{1n},W_{2n},\ldots,W_{nn})^\top,
	\end{equation*} 
	and $W_{in}$ is the algebraic cofactor of the $(i,n)$-entry for the ordinary walk matrix $W(G)$. A key property of the modified walk matrix is that it is an integer matrix and always has a nonzero determinant. The following integer invariant associated with a graph $G\in \mathcal{H}_n$ is crucial.
	\begin{definition}[\cite{Qiu2022}] Let $G\in \mathcal{H}_n$. Define
		$
		\ell_0=\ell_0(G)=\frac{\eta^\top \eta}{2}
		$, where
		\begin{equation*}
			\eta= \frac{\xi(G)}{\gcd(W_{1n},\ldots,W_{nn})}.
		\end{equation*}
	\end{definition}
	
	Note that since $\rank W(G) = n-1$, its left null space $\mathcal{N}(W(G)^\top)$ is one-dimensional. By the properties of algebraic cofactors, $\xi(G)$ spans this null space. Consequently, dividing $\xi(G)$ by the greatest common divisor of its entries ensures that $\eta$ is precisely the coprime integer vector generating $\mathcal{N}(W(G)^\top)$.
	
	 It was proved in \cite{Qiu2022} that for an almost controllable graph $G$, the number $\ell_0(G)$ is a positive integer, moreover $\ell_0(G)=1$ if and only if $G$ has a nontrivial automorphism. For a prime $p$, we use $\mathbb{F}_p$ to denote the finite field of order $p$, and $\bar{\mathbb{F}}_p$ to denote the algebraic closure of the finite field $\mathbb{F}_p$. Let $J_n$ (or $J$) denote  the all-ones matrix of order $n$.  For a graph $G$ with adjacency matrix $A$, the \emph{invariant polynomial} \cite{Wang2023} of $G$, with respect to a prime number $p$, is defined as 
	\begin{equation*}
		\Phi_p(G;x)=\gcd(\chi(A;x),\chi(A+J;x))\in \mathbb{F}_p[x].
	\end{equation*}
	We remark that $\Phi_p(G;x)$ is invariant under generalized cospectrality, that is, if $G$ and $H$ are generalized cospectral, then $\Phi_p(G;x)=\Phi_p(H;x)$.
	
	 For every $n\times n$ integer matrix $M$ with full rank, there exist two unimodular matrices $U$ and $V$ such that $M=USV$, where $S=\diag(d_1,\ldots,d_n)$ with $d_i\mid d_{i+1}$ for $1\le i\le n-1$; the diagonal matrix $S$ is known as the \emph{Smith normal form} (SNF) of $M$, and $d_i$ is called the $i$-th \emph{invariant factor} of $M$. For two nonzero integers $a$ and $b$, we use $a\mid\mid b$ to denote that $a$ precisely divides $b$, i.e., $a\mid b$ but $a^2\nmid b$. Now we can state the DGS criterion for almost controllable graphs due to Lin et al.~\cite{Lin2026}, which simplifies and improves the original criterion obtained by Wang et al. \cite{Wang2021} that only works for graphs $G\in \mathcal{H}_n$ with $\ell_0(G) =1$.
	
	\begin{theorem}[\cite{Lin2026}]\label{lnc}
		Let $G\in \mathcal{H}_n$ and $\widetilde{d}_n$ be the last invariant factor of $\widetilde{W}(G)$. Suppose the following conditions hold:\\
		\noindent\textup{(i)} $\widetilde{d}_n$ is square-free;\\
		\noindent\textup{(ii)} $\ell_0$ is either $1$ or an odd prime with $\ell_0\mid\mid \det \widetilde{W}$;\\
		\noindent \textup{(iii)} $\Phi_p(G;x)$ has exactly two distinct roots over $\bar{\mathbb{F}}_p$ for any odd prime factor $p$ of $\widetilde{d}_n$ with $p\neq \ell_0$.
		
		Then $G$ is DGS.
	\end{theorem}
	
	For a graph $G\in \mathcal{H}_n$, let $\{\lambda_1,\ldots,\lambda_{n-1}\}\cup \{\lambda_0\}$ be the partition of the spectrum of $G$ such that $\lambda_1,\ldots,\lambda_{n-1}$ are main eigenvalues; the remaining unique eigenvalue $\lambda_0$ is the plain eigenvalue \cite{Hayat2017} (or non-main eigenvalue if $G$ has no multiple eigenvalues). The main polynomial  \cite{Cvetkovic2010} of $G$ is 
	\begin{equation*}
		m(G;x)=(x-\lambda_1)(x-\lambda_2)\cdots(x-\lambda_{n-1}).
	\end{equation*} 
	It is known \cite{Rowlinson2007} that $m(G;x)=\chi(A|_U;x)$, where $U=\col W(G)$, the column space of $W(G)$.
	
	The main result of this paper is the following theorem, which can be regarded as an extension of Theorem \ref{lc} from controllable graphs to almost controllable graphs.
	
	\begin{theorem}\label{main}
		Let $G_0\in \mathcal{H}_{n_0}$ satisfy the conditions of Theorem \ref{lnc}. Let $G_1,G_2,\ldots$ be a sequence of graphs defined by 
		\begin{equation*}
			G_i=\begin{cases}
				G_{i-1}\cup K_1,&\text{if $i$ is odd,}\\
				G_{i-1}\vee K_1,&\text{if $i$ is even.}
			\end{cases}
		\end{equation*}
		If the main polynomials satisfy 
		\begin{equation*}
			(|m(G_0;0)|,|m (G_1^c;0)|)=\begin{cases}
				(1,2),&\text{if $n_0$ is even,}\\
				(2,1),&\text{if $n_0$ is odd,}
			\end{cases}
		\end{equation*}
		then each graph $G_i$ satisfies the conditions of Theorem \ref{lnc} and hence is DGS.
	\end{theorem}
	
	The proof of Theorem \ref{main} will be given in Section \ref{pf}. We end this section by introducing some equivalent but more convenient descriptions of the conditions in Theorem \ref{lnc}.
	
	\begin{theorem}[\cite{Lin2026}]\label{equivC}
		Let $G\in \mathcal{H}_n$ and $\widetilde{d}_n$ be the last invariant factor of $\widetilde{W}(G)$. Suppose $\ell_0$ is either 1 or an odd prime with $\ell_0\mid\mid \det \widetilde{W}(G)$. Then the following statements are equivalent:\\
		\noindent\textup{(i)} $\widetilde{d}_n$ is square-free;\\	
		\noindent\textup{(ii)} the SNF of $\widetilde{W}(G)$ is 
		\begin{equation*}
			\widetilde{S}=\diag(\underbrace{1,1,\ldots,1}_{\lceil\frac{n}{2}\rceil}, \underbrace{2,2,\ldots,2,2b,2\ell_0 b}_{\lfloor\frac{n}{2}\rfloor}),
		\end{equation*}
		where $b$ is odd and square-free;\\
		\noindent\textup{(iii)} $\det \widetilde{W}(G)=\pm 2^{\lfloor\frac{n}{2}\rfloor}\ell_0b^2$, where $b$ is odd and square-free.
	\end{theorem}
	
	We remark that only the implication (i)$\implies$(ii) was explicitly stated in \cite[Theorem 7]{Lin2026}. The implications (ii)$\implies$(i) and (ii)$\implies$(iii) are straightforward; and the implication (iii)$\implies$(ii) can be proved in a similar way using \cite[Lemmas 4 and 7]{Lin2026}; see \cite{Lin2026} for details. 
	
	Suppose $G\in \mathcal{H}_n$ satisfies the first two conditions of Theorem \ref{lnc}. Let $p$ be any odd prime factor of $\widetilde{d}_n$ (or equivalently of $\det \widetilde{W}(G)$) with $p\neq \ell_0$. Then by Theorem \ref{equivC} (ii), we find that the nullity of $\widetilde{W}(G)$ over $\mathbb{F}_p$ is exactly two. Since the number of the roots of $\Phi_p(G;x)$ in $\bar{\mathbb{F}}_p$ is upper bounded by the nullity of $\widetilde{W}(G)$ over $\mathbb{F}_p$ (see \cite[Lemma 14]{Lin2026}), we find that $\Phi_p(G;x)$ has at most two distinct roots. This means that we may replace the last condition of Theorem \ref{lnc} with the following equivalent, more readily verifiable criterion:
	
	\noindent(iii') $\Phi_p(G;x)$ has \emph{at least} two distinct roots over $\bar{\mathbb{F}}_p$ for any odd prime factor $p$ of $\widetilde{d}_n$ with $p\neq \ell_0$.
	
	Now, we can give an equivalent form of Theorem \ref{lnc}, a form that will be used in proving our main theorem (Theorem \ref{main}).
	
	\begin{theorem}\label{lnc2}
		Let $G\in \mathcal{H}_n$. Suppose the following conditions hold:\\
		\noindent\textup{(i)} $\det \widetilde{W}(G)=\pm 2^{\lfloor\frac{n}{2}\rfloor}\ell_0p_1^2 p_2^2 \cdots p_m^2$, where $m\ge 0$ and $p_i$'s are $m$ distinct odd primes;
		\\
		\noindent\textup{(ii)} $\ell_0$ is either $1$ or an odd prime with $\ell_0\not\in\{p_1,\ldots,p_m\}$;\\
		\noindent \textup{(iii)} $\Phi_{p}(G;x)$ has at least two distinct roots over $\bar{\mathbb{F}}_{p}$ for any $p\in\{p_1,p_2,\ldots,p_m\}$.
		
		Then $G$ is DGS.
	\end{theorem}
	
	\section{Complement-invariant property}
	
	Let $\mathcal{P}$ be a graph property, that is, $\mathcal{P}$ is a Boolean-valued function on the set of all simple graphs satisfying the isomorphic-invariant condition: if $G$ is isomorphic to $H$, then $\mathcal{P}(G)=\mathcal{P}(H)$. We say that a graph $G$ possesses property $\mathcal{P}$ if $\mathcal{P}(G)=1$. A property $\mathcal{P}$ is called \emph{complement-invariant} if $\mathcal{P}(G)=\mathcal{P}(G^c)$ for any graph $G$. Note that a graph $G$ is DGS if and only if its complement $G^c$ is DGS. Thus, the DGS-property is complement-invariant.
	
	In \cite{Liu2019}, Liu et al. proved that for any simple graph $G$, there exists a unimodular upper-triangular integer matrix $U$ such that 
	\begin{equation}\label{dwc}
		W(G^c)=W(G)U,
	\end{equation} 
	which implies $\det W(G^c)=\pm \det W(G)$. Let $\mathcal{P}_1$ be the property described as the condition of Theorem \ref{wt}. That is, an $n$-vertex graph has property $\mathcal{P}_1$ if $2^{-\lfloor n/2\rfloor}\det W(G)$ is odd and square-free. According to Eq.~\eqref{dwc}, we know that the property $\mathcal{P}_1$ is complement-invariant. Now, we restrict ourselves to $\mathcal{H}_n$, the family of almost controllable graphs of order $n$. It is known that $G\in \mathcal{H}_n$ if and only if $G^c\in \mathcal{H}_n$. Let $\mathcal{P}_2$ be the property described as the three conditions in Theorem \ref{lnc2}. 
	
	\begin{theorem}
		The property $\mathcal{P}_2$ is complement-invariant.
	\end{theorem}
	
	\begin{proof}
		Let $G$ be any graph in $\mathcal{H}_n$ that possesses $\mathcal{P}_2$. We need to show that $G^c$ also possesses $\mathcal{P}_2$. Let $U$ be the unimodular upper-triangular matrix $U$ such that Eq.~\eqref{dwc} holds. Let $U'$ be the matrix obtained from $U$ by deleting the last column. Then $U'$ has the form 
		\begin{equation*}
			U'=\begin{pmatrix}
				U_1\\0
			\end{pmatrix},
		\end{equation*}
		where $U_1$ is an $(n-1)\times (n-1)$ unimodular upper-triangular matrix. Let $W(i,n)$ and $W^c(i,n)$ be the $(n-1)\times (n-1)$ submatrices obtained from $W(G)$ and $W(G^c)$ by deleting the $i$-th row as well as the last column, respectively. Noting that the last row of $U'$ is the zero vector, it is easy to see from Eq.~\eqref{dwc} that 
		\begin{equation*}
			W^c(i,n)=W(i,n)U_1.
		\end{equation*}
		It follows that  
		\begin{equation}\label{xx}
			\xi(G^c)=\begin{pmatrix}
				(-1)^{1+n}\det W^c(1,n)\\
				(-1)^{2+n}\det W^c(2,n)\\
				\vdots\\
				(-1)^{n+n}\det W^c(n,n)
			\end{pmatrix}=(\det U_1)
			\begin{pmatrix}
				(-1)^{1+n}\det W(1,n)\\
				(-1)^{2+n}\det W(2,n)\\
				\vdots\\
				(-1)^{n+n}\det W(n,n)
			\end{pmatrix}=\pm \xi(G).
		\end{equation}
		Let 
		\begin{equation*}
			U_2=\begin{pmatrix}
				U_1&0\\
				0&\delta
			\end{pmatrix},
		\end{equation*}
		where $\delta=\det U_1=\pm 1$. Clearly $U_2$ is a unimodular matrix. From Eqs.~\eqref{dwc} and \eqref{xx}, we easily see that the modified walk matrices satisfy
		\begin{equation*}
			\widetilde{W}(G^c)=\widetilde{W}(G)U_2,
		\end{equation*}
		which implies $\det \widetilde{W}(G^c)=\pm \det\widetilde{W}(G)$. This means that Condition (i) holds for $G^c$. Furthermore, by Eq.~\eqref{xx}, we see that $\eta(G)=\pm \eta(G^c)$ and hence $\ell_0(G)=\ell_0(G^c)$. Noting that $\det \widetilde{W}(G^c)=\pm \det\widetilde{W}(G)$, Condition (ii) holds for $G^c$. 
		
		It remains to show that $G^c$ satisfies Condition (iii), i.e., $\Phi_p(G^c)$ has at least two distinct roots over $\bar{\mathbb{F}}_p$ for any $p\in\{p_1,\ldots,p_m\}$. We claim that if $\lambda$ is a root of $\Phi_p(G;x)$ then $-1-\lambda$ is a root of $\Phi_p(G^c;x)$. Let $\lambda\in \bar{\mathbb{F}}_p$ be a root of $\Phi_p(G;x)$, that is, $\lambda$ is a common root eigenvalue of $A$ and $A+J$. Then there exist two nonzero vectors $\alpha$ and $\beta$ such that $A\alpha=\lambda \alpha$ and $(A+J)\beta=\lambda \beta$. Taking the transpose and noting that $A$ is symmetric, we have
		\begin{equation}\label{a1}
			\alpha^\top (\lambda I-A)\beta=\beta^\top (\lambda I-A)\alpha=0.
		\end{equation}
		On the other hand, direct computations show that
		\begin{equation}\label{a2}
			\alpha^\top (\lambda I-A)\beta=\alpha^\top J\beta=\alpha^\top e e^\top \beta=(e^\top \alpha)(e^\top \beta).
		\end{equation}
		Combining Eq.~\eqref{a1} and Eq.~\eqref{a2} leads to $e^\top \alpha =0$ or $e^\top \beta=0$. Let $\gamma=\alpha$ if $e^\top \alpha=0$; and $\gamma=\beta$ otherwise. Then it is easy to see that $J\gamma=0$ and $A\gamma=(A+J)\gamma=\lambda \gamma$. It follows that
		\begin{equation*}
			(A(G^c)+J)\gamma=A(G^c)\gamma = (J - I - A)\gamma= (-1 - \lambda)\gamma.
		\end{equation*}
		This demonstrates that $-1 - \lambda$ is a root of $\Phi_p(G^c; x)$ and hence the claim follows. Since $\Phi_p(G; x)$ has at least two distinct roots, $\Phi_p(G^c; x)$ will also possess at least two distinct roots. Thus, Condition (iii) holds for $G^c$. This completes the proof.
	\end{proof}	
	
	\section{Evolution of  modified walk matrices}
	
	For a controllable graph $G$, Liu et al.~\cite{Liu2019} proved that 
	\begin{equation*}
		\det W(G\cup K_1)=\pm \det A(G)\det W(G),
	\end{equation*}
	which implies that $G\cup K_1$ remains controllable if and only if $A$ is nonsingular, i.e., $G$ does not contain $0$ as an eigenvalue. The following result gives a natural extension of this result to almost controllable graphs.
	
	\begin{theorem} \label{wunion}
		Let $G\in \mathcal{H}_n$. Then $G\cup K_1$ remains almost controllable if and only if $m(G;0) \neq 0$, i.e., $G$ does not contain $0$ as a main eigenvalue. Furthermore, we have
		\begin{equation} \label{dw_union}
			\det \widetilde{W}(G\cup K_1) = \pm2^{\lfloor n/2 \rfloor - \lfloor (n+1)/2 \rfloor} (m(G;0))^2 \det \widetilde{W}(G).
		\end{equation}
	\end{theorem}
	
	\begin{proof}
		Let $G'=G\cup K_1$. Without loss of generality, we label the newly added isolated vertex as the first vertex of $G'$. The adjacency matrix $A' \in \mathbb{R}^{(n+1) \times (n+1)}$ and the all-ones vector $e' \in \mathbb{R}^{n+1}$ of $G'$ can be written in block form as:
		\begin{equation*}
			A' = \begin{pmatrix} 0 & 0^\top \\ 0 & A \end{pmatrix}, \quad e' = \begin{pmatrix} 1 \\ e \end{pmatrix}.
		\end{equation*}
		Consequently, we obtain
		\begin{equation} \label{wgp}
			W(G') = [e', A'e', \dots, (A')^n e'] = \begin{pmatrix} 1 & \mathbf{0}^\top \\ e & A W(G) \end{pmatrix}.
		\end{equation}
		From Eq.~\eqref{wgp}, we have $\rank W(G') = 1 + \rank (A W(G))$. Note that $\rank W(G)=n-1$ and $\chi(A|_{\col W(G)};x)=m(G;x)$. The rank of $A W(G)$ is equal to $n-1$ if and only if the restriction of $A$ to $\col W(G)$ is non-singular. Clearly, this non-singularity is equivalent to $m(G;0) \neq 0$. Thus, $G'$ is almost controllable if and only if $m(G;0) \neq 0$. Now, assuming $m(G;0) \neq 0$, we analyze the modified walk matrix $\widetilde{W}(G')$. Recall that $\xi(G)$ is the vector whose entries are the algebraic cofactors of the last column of $W(G)$. It is known \cite{Wang2021} that $A(G)\xi(G)=\lambda_0 \xi(G)$, where $\lambda_0$ is the unique plain eigenvalue of $G$. We assume further that $\lambda_0\neq 0$ and the case $\lambda_0=0$ can be settled by a standard argument of continuous perturbation. Let $\lambda_1, \dots, \lambda_{n-1}$ be the $n-1$ main eigenvalues of $G$. Then we have 
		\[ \det A=\lambda_0\prod_{i=1}^{n-1}\lambda_i=(-1)^{n-1} \lambda_0m(G;0)\neq 0. \]
		As $A$ is invertible over $\mathbb{R}$, we find that 
		\begin{equation}\label{adjv}
			\adj(A)\xi(G)=(\det A) (A^{-1}\xi(G))=(\det A) \frac{1}{\lambda_0}\xi(G)=(-1)^{n-1} m(G;0) \xi(G).
		\end{equation}
		
		Let $W'_{i,n+1}$ be the algebraic cofactor of the $(i,n+1)$-entry of $W(G')$. Then, by definition, $\xi(G') = (W'_{1,n+1},W'_{2,n+1},\ldots,W'_{n+1,n+1})^\top$. By Eq.~\eqref{wgp}, we easily see that the first entry satisfies $W'_{1,n+1}=(-1)^{n+2}\det W(G)=0$. For the remaining entries, expanding the cofactors along the first row reveals that $W'_{i,n+1}$, with $i\ge 2$, is exactly the cofactor of the $(i-1,n)$-entry of the submatrix $A W(G)$. Thus, we have 
		\begin{equation*}
			(W'_{2,n+1},\ldots,W'_{n+1,n+1})^\top= (\adj (AW(G)))^\top e_n,
		\end{equation*}
		where $e_n=(0,\ldots,0,1)^\top\in \mathbb{R}^n$.
		Recall the property of the adjugate matrix: $\operatorname{adj}(XY) = \operatorname{adj}(Y)\operatorname{adj}(X)$. Noting that $\adj A$ is symmetric and using Eq.~\eqref{adjv}, we obtain
		\begin{equation*}
			(W'_{2,n+1},\ldots,W'_{n+1,n+1})^\top=\adj(A)(\adj W)^\top e_n=\adj(A)\xi(G)= (-1)^{n-1} m(G;0) \xi(G).
		\end{equation*}
		It follows that
		\begin{equation*}
			\xi(G') = \begin{pmatrix} 0 \\ (-1)^{n-1} m(G;0) \xi(G) \end{pmatrix}.
		\end{equation*}
		This, together with Eq.~\eqref{wgp}, indicates that the modified walk matrix of $G'$ is 
		\begin{equation*}
			\widetilde{W}(G') = \begin{pmatrix} 1 & 0 & \dots & 0 & 0 \\ e & Ae & \dots & A^{n-1}e & 2^{1 - \lfloor \frac{n+1}{2} \rfloor} (-1)^{n-1} m(G;0) \xi(G) \end{pmatrix}.
		\end{equation*}
		Expanding the determinant of $\widetilde{W}(G')$ along the first row, we get $\det \widetilde{W}(G') = \det(B)$, where $B$ is the $n \times n$ lower-right block:
		\begin{equation*} 
			B = \left[ Ae, A^2e, \dots, A^{n-1}e, \ \ 2^{1 - \lfloor \frac{n+1}{2} \rfloor} (-1)^{n-1} m(G;0) \xi(G) \right].
		\end{equation*}
		On the other hand, multiplying $\widetilde{W}(G)$ by $A$ and noting that $A\xi(G)=\lambda_0\xi(G)$, we have
		\begin{equation*}
			A\widetilde{W}(G)= \left[ Ae, A^2e, \dots, A^{n-1}e, \ \ \lambda_0 2^{1 - \lfloor \frac{n}{2} \rfloor} \xi(G) \right].
		\end{equation*}
		Since the first $n-1$ columns of $B$ and $A\widetilde{W}(G)$ are identical, by the multilinearity of determinants with respect to the last column, we obtain
		\begin{equation}\label{cp}
			\frac{\det \widetilde{W}(G')}{\det(A\widetilde{W}(G))} = \frac{\det B}{\det(A\widetilde{W}(G))} = \frac{2^{1 - \lfloor \frac{n+1}{2} \rfloor} (-1)^{n-1} m(G;0)}{\lambda_0 2^{1 - \lfloor \frac{n}{2} \rfloor}} = 2^{\lfloor \frac{n}{2} \rfloor - \lfloor \frac{n+1}{2} \rfloor} \frac{(-1)^{n-1} m(G;0)}{\lambda_0}.
		\end{equation}
		Since \[ \det(A\widetilde{W}(G)) = \det(A)\det(\widetilde{W}(G))=\lambda_0(-1)^{n-1}m(G;0)\det \widetilde{W}(G), \] we can rewrite Eq.~\eqref{cp} as 
		\begin{equation*}
			\det \widetilde{W}(G') = 2^{\lfloor \frac{n}{2} \rfloor - \lfloor \frac{n+1}{2} \rfloor} (m(G;0))^2 \det \widetilde{W}(G).
		\end{equation*}
		This establishes Eq.~\eqref{dw_union} under the assumption that $\lambda_0 \neq 0$.
		
		Finally, suppose that $\lambda_0 = 0$. We consider a continuous perturbation of the adjacency matrix $A(\epsilon) = A + \epsilon I$ for $\epsilon > 0$, where $I$ is the identity matrix. Since the main eigenvalues, plain eigenvalues and the determinant of the modified walk matrix depend continuously on the entries of the matrix, the relation in Eq.~\eqref{dw_union} holds for $A(\epsilon)$ for all $\epsilon > 0$. Taking the limit as $\epsilon \to 0^+$ on both sides, we obtain:
		\begin{equation*}
			\det \widetilde{W}(G') = 2^{\lfloor n/2 \rfloor - \lfloor (n+1)/2 \rfloor} (m(G;0))^2 \det \widetilde{W}(G),
		\end{equation*}
		which strictly completes the proof for all almost controllable graphs. This completes the proof of Theorem \ref{wunion}.	
	\end{proof}
	
	Recall that the determinant of the modified walk matrix is an invariant (up to a sign) under graph complement. Since $G\vee K_1=((G^c \cup K_1))^c$, the following counterpart of Theorem \ref{wunion} for disjoint union is immediate.
	
	\begin{theorem} \label{wjoin}
		Let $G\in \mathcal{H}_n$. Then $G\vee K_1$ remains almost controllable if and only if $m(G^c;0) \neq 0$, i.e., $G^c$ does not contain $0$ as a main eigenvalue. Furthermore, we have
		\begin{equation*} 
			\det \widetilde{W}(G\vee K_1) = \pm2^{\lfloor n/2 \rfloor - \lfloor (n+1)/2 \rfloor} (m(G^c;0))^2 \det \widetilde{W}(G).
		\end{equation*}
	\end{theorem}
	
	\section{Proof of Theorem \ref{main}}\label{pf}
	
	Recall that $\mathcal{P}_2$ is the graph property that is described by the three conditions of Theorem \ref{lnc2}. It is known that $\mathcal{P}_2$ is complement-invariant and that graphs with this property are DGS. Now, we consider two graph properties $\mathcal{P}_3$ and $\mathcal{P}_4$, which are stronger than $\mathcal{P}_2$. Precisely, we say that an $n$-vertex graph $G$ possesses property $\mathcal{P}_3$ if $G$ is an almost controllable graph satisfying $\mathcal{P}_2$ and the equation
	\begin{equation}\label{mg}
		(| m(G;0)|,|m((G\cup K_1)^c;0)|)=\begin{cases}
			(1,2),&\text{if }n\text{ is even},\\
			(2,1),&\text{if }n\text{ is odd}.
		\end{cases}
	\end{equation}
	Similarly, we say that a graph $G$ possesses property $\mathcal{P}_4$ if $G$ is an almost controllable graph satisfying $\mathcal{P}_2$ and the equation
	\begin{equation}\label{mgc}
		(| m(G^c;0)|,| m((G^c\cup K_1)^c;0)|)=\begin{cases}
			(1,2),&\text{if }n\text{ is even},\\
			(2,1),&\text{if }n\text{ is odd}.
		\end{cases}
	\end{equation}
	Clearly, the properties $\mathcal{P}_3$ and $\mathcal{P}_4$ are dual properties under graph complementation, that is, $G$ has property $\mathcal{P}_3$ if and only if $G^c$ has property $\mathcal{P}_4$. A key result of this section is that if $G$ possesses $\mathcal{P}_3$ then the disjoint union $G\cup K_1$ must possess $\mathcal{P}_4$. Theorem \ref{main} follows by alternately applying this result and its dual.
	
	\begin{lemma}\label{mgh}
		Let $G\in \mathcal{H}_n$ be such that both $m(G;0)$ and $m((G\cup K_1)^c;0)$ are nonzero. Let $H=(G\cup K_1)\vee K_1$. Then $H\in \mathcal{H}_{n+2}$ and $m(G;0)=\pm m(H;0)$.
	\end{lemma}
	\begin{proof}
		Since $G\in \mathcal{H}_n$ and $m(G;0)\neq 0$, Theorem \ref{wunion} implies that $G\cup K_1$ is almost controllable. Consequently, using Theorem \ref{wjoin} for the graph $G'=G\cup K_1$ and noting $m((G')^c;0)\neq 0$, we find that $G'\vee K_1$ is almost controllable, i.e., $H\in \mathcal{H}_{n+2}$. Then, by labeling the vertices of $H$ appropriately, we may write the adjacency matrix of $A(H)$ in the form
		\begin{equation*}
			A(H)=\begin{pmatrix}
				0&1&e^\top\\
				1&0&0^\top\\
				e&0&A(G)
			\end{pmatrix}.
		\end{equation*}
		Expanding the determinant of $A(H)$ by the second row and then by the first column, we find that 
		\begin{equation}\label{gh}
			\det A(H)=-\det A(G).
		\end{equation}
		Let $\lambda_0$ be the unique plain eigenvalue of $G$ and $\eta$ be the coprime vector in $\mathcal{N}(W^\top(G))$. Then $e^\top \eta=0$ and $A(G)\eta=\lambda_0 \eta$. Let 
		\begin{equation*}
			\eta'=\begin{pmatrix}
				0\\
				0\\
				\eta
			\end{pmatrix}\in \mathbb{Z}^{n+2}.
		\end{equation*}
		Then we have
		\begin{equation*}
			A(H)	\eta'=\begin{pmatrix}
				0&1&e^\top\\
				1&0&0^\top\\
				e&0&A(G)
			\end{pmatrix}\begin{pmatrix}
				0\\
				0\\
				\eta
			\end{pmatrix}
			=\begin{pmatrix}
				e^\top \eta\\
				0\\
				A(G)\eta
			\end{pmatrix}=
			\begin{pmatrix}
				0\\
				0\\
				\lambda_0\eta
			\end{pmatrix}=\lambda_0\eta'.
		\end{equation*}
		Noting that the sum of entries in $\eta'$ is zero, we see that $\lambda_0$ is also the plain eigenvalue of $H$. If $\lambda_0\neq 0$, then, by Eq.~\eqref{gh},
		\begin{equation*}
			m(G;0)=\pm \frac{\det A(G)}{\lambda_0}=\pm \frac{\det A(H)}{\lambda_0}=\pm m(H;0).
		\end{equation*}
		To show $m(G;0) = \pm m(H;0)$ for the case $\lambda_0 = 0$, we introduce the perturbed matrix 
		\begin{equation*}
			B(\epsilon) = \begin{pmatrix}
				0 & 1 & e^\top \\
				1 & 0 & 0^\top \\
				e & 0 & A(G) + \epsilon I_n
			\end{pmatrix},
		\end{equation*}
		where only the $A(G)$ block is perturbed. Let $A(\epsilon) = A(G) + \epsilon I_n$. We can easily verify that $\eta'$ is still a plain eigenvector for $B(\epsilon)$ with eigenvalue $\epsilon$, since
		\begin{equation*}
			B(\epsilon)\eta' = 
			\begin{pmatrix}
				0 & 1 & e^\top \\
				1 & 0 & 0^\top \\
				e & 0 & A(\epsilon)
			\end{pmatrix}
			\begin{pmatrix}
				0 \\ 0 \\ \eta
			\end{pmatrix}
			= \begin{pmatrix}
				e^\top \eta \\ 0 \\ A(\epsilon)\eta
			\end{pmatrix}
			= \begin{pmatrix}
				0 \\ 0 \\ \epsilon \eta
			\end{pmatrix}
			= \epsilon \eta'.
		\end{equation*}
		Since $e^\top \eta' = 0$, the vector $\eta'$ is orthogonal to the column space of the walk matrix $W(B(\epsilon))$, which implies $\rank W(B(\epsilon)) \le (n+2) - 1 = n+1$. On the other hand, the rank of a matrix is a lower semi-continuous function of its entries. Because $\rank W(B(0)) = \rank W(H) = n+1$, we must have $\rank W(B(\epsilon)) \ge n+1$ for sufficiently small $\epsilon$. Therefore, $\rank W(B(\epsilon)) = n+1$ in a neighborhood of $\epsilon = 0$, meaning that $B(\epsilon)$ has exactly $(n+1)$ main eigenvalues and hence $\epsilon$ is the unique plain eigenvalue. Expanding the determinant of $B(\epsilon)$ by the second row, we obtain 
		\begin{equation*}
			\det B(\epsilon) = - \det A(\epsilon).
		\end{equation*}
		It follows that
		\begin{equation*}
			m(B(\epsilon); 0) = \pm \frac{\det B(\epsilon)}{\epsilon} = \pm \frac{\det A(\epsilon)}{\epsilon} = \pm m(A(\epsilon); 0),
		\end{equation*}
		where the sign $\pm$ only depends on $n$ and is independent of $\epsilon$. Taking the limit as $\epsilon \to 0$ on both sides, we obtain
		\begin{equation*}
			m(H; 0) =\pm \lim_{\epsilon \to 0} m(B(\epsilon); 0) = \pm \lim_{\epsilon \to 0} (m(A(\epsilon); 0)) = \pm m(G; 0).
		\end{equation*}
		This establishes $m(G; 0) = \pm m(H; 0)$ for all cases, completing the proof.
	\end{proof}

	\begin{lemma} \label{lem:l0}
		Let $G\in \mathcal{H}_n$ have property $\mathcal{P}_3$. Then $G\cup K_1$ has property $\mathcal{P}_4$.
	\end{lemma}
	\begin{proof}
		Let $G' = G \cup K_1$ and we label the newly added vertex as the first vertex of $G'$. By Eq.~\eqref{mg}, we know that $m(G;0)\neq 0$ and hence $G'$ is almost controllable by Theorem \ref{wunion}. Let $\eta$ be the coprime vector in $\mathcal{N}(W(G)^\top)$. Let $\lambda_0$ be the unique plain eigenvalue of $G$. Noting that $A(G)\eta=\lambda_0\eta$ and using Eq.~\eqref{wgp}, it is not difficult to see that 
		\begin{equation*}
			W(G')^\top\begin{pmatrix} 0\\\eta\end{pmatrix}=\begin{pmatrix}
				1&e^\top\\
				0&W(G)^\top A 
			\end{pmatrix}\begin{pmatrix} 0\\\eta\end{pmatrix}=\begin{pmatrix}
				e^\top \eta\\\lambda_0 W(G)^\top \eta
			\end{pmatrix}=0.
		\end{equation*}
		Let $\eta'=\begin{pmatrix} 0\\\eta\end{pmatrix}\in \mathbb{Z}^{n+1}$. Since the elements of $\eta$ are coprime, the elements of $\eta'$ are also coprime, meaning $\eta'$ is the coprime integral generator of $\mathcal{N}(W(G')^T)$. Consequently, we have 
		\[ \ell_0(G') = \frac{(\eta')^\top \eta'}{2} = \frac{\eta^\top \eta}{2} = \ell_0(G). \]
		
		We proceed to check that the number $\det \widetilde{W}(G')$ has the desired factorization as described in Theorem \ref{lnc2}. Let 
		\[ \det \widetilde{W}(G)=\pm 2^{\lfloor\frac{n}{2}\rfloor}\ell_0p_1^2p_2^2\ldots p_m^2 \]
		be the factorization of $\det \widetilde{W}(G)$, where $m\ge 0$ and $p_i$'s are $m$ distinct odd primes. Then, by Eq.~\eqref{dw_union} of Theorem \ref{wunion}, we have
		\begin{equation*} 
			\det \widetilde{W}(G') = \pm2^{2\lfloor \frac{n}{2} \rfloor - \lfloor \frac{n+1}{2} \rfloor} (m(G;0))^2 \ell_0p_1^2p_2^2\ldots p_m^2.
		\end{equation*}
		By Eq.~\eqref{mg}, we have 
		\begin{equation*}
			m(G;0)=\begin{cases}
				\pm1,&\text{if }n\text{ is even,}\\
				\pm2,&\text{if }n\text{ is odd.}
			\end{cases}
		\end{equation*}
		Using a simple discussion on the parity of $n$, we find that 
		\begin{equation*} 
			2^{2\lfloor \frac{n}{2} \rfloor - \lfloor \frac{n+1}{2} \rfloor} (m(G;0))^2=2^{\lfloor\frac{n+1}{2}\rfloor},
		\end{equation*}
		which means that $\det \widetilde{W}(G')$ has the desired factorization as described in Theorem \ref{lnc2} (i). Next, we check Condition (iii) of Theorem \ref{lnc2} for $G'$. Let $p$ be any of the odd prime factors $p_1, \dots, p_m$. Since $\Phi_p(G; x)$ has at least two distinct roots over $\bar{\mathbb{F}}_p$, it suffices to show that any root of $\bar{\mathbb{F}}_p$ is also a root of $\Phi_p(G'; x)$. Let $\lambda \in \bar{\mathbb{F}}_p$ be any root of $\Phi_p(G; x)$. As established in Section 2, there exists a common eigenvector $\gamma \in \bar{\mathbb{F}}_p^n$ of $A(G)$ and $A(G) + J_n$ corresponding to $\lambda$ such that $e^\top \gamma = 0$. Let $\gamma' = \begin{pmatrix} 0 \\ \gamma \end{pmatrix} \in \bar{\mathbb{F}}_p^{n+1}$. Since $e^\top \gamma = 0$, we have $(e')^\top \gamma' = 0$, where $e'$ is the all-ones vector of dimension $n+1$. Direct computation yields
		\begin{equation*}
			A(G') \gamma' = \begin{pmatrix} 0 & 0^\top \\ 0 & A(G) \end{pmatrix} \begin{pmatrix} 0 \\ \gamma \end{pmatrix} = \begin{pmatrix} 0 \\ A(G)\gamma \end{pmatrix} = \lambda \gamma',
		\end{equation*}
		and
		\begin{equation*}
			(A(G') + J_{n+1}) \gamma' = \begin{pmatrix} 1 & e^\top \\ e & A(G) + J_n \end{pmatrix} \begin{pmatrix} 0 \\ \gamma \end{pmatrix} = \begin{pmatrix} e^\top \gamma \\ (A(G)+J_n)\gamma \end{pmatrix} = \begin{pmatrix} 0 \\ \lambda \gamma \end{pmatrix} = \lambda \gamma'.
		\end{equation*}
		This indicates that $\lambda$ is also a common eigenvalue of $A(G')$ and $A(G') + J_{n+1}$ over $\bar{\mathbb{F}}_p$. Thus, every root of $\Phi_p(G; x)$ is also a root of $\Phi_p(G'; x)$ over $\bar{\mathbb{F}}_p$. Hence, $G'$ satisfies Condition (iii) of Theorem \ref{lnc2}.
		
		It remains to check Eq.~\eqref{mgc} for the graph $G'$ on $n+1$ vertices, which requires showing
		\begin{equation}\label{mgc2}
			(| m((G')^c;0)|,| m(((G')^c\cup K_1)^c;0)|)=\begin{cases}
				(1,2),&\text{if }n+1\text{ is even},\\
				(2,1),&\text{if }n+1\text{ is odd}.
			\end{cases}
		\end{equation}
		Notice that $G' = G \cup K_1$, yielding $(G')^c = (G \cup K_1)^c$. Furthermore, $((G')^c \cup K_1)^c = G' \vee K_1 = (G \cup K_1) \vee K_1$. By Lemma \ref{mgh}, we have $m(((G')^c \cup K_1)^c; 0) = \pm m(G; 0)$. Consequently, the left side of Eq.~\eqref{mgc2} is exactly the swapped pair of $G$:
		\begin{equation*}
			(| m((G')^c;0)|,| m(((G')^c\cup K_1)^c;0)|) = (| m((G \cup K_1)^c;0)|, |m(G;0)|).
		\end{equation*}
		Since $G$ possesses property $\mathcal{P}_3$, Eq.~\eqref{mg} holds for $n$. Recognizing that the parity of $n+1$ is strictly opposite to that of $n$, this swapped pair trivially satisfies Eq.~\eqref{mgc2}. This completes the proof that $G'$ possesses property $\mathcal{P}_4$.
	\end{proof}
	Dually, we obtain the corresponding result for the join operation.
	
	\begin{lemma} \label{lem:l0_dual}
		Let $G \in \mathcal{H}_n$ have property $\mathcal{P}_4$. Then $G \vee K_1$ has property $\mathcal{P}_3$.
	\end{lemma}
	\begin{proof}
		Suppose $G \in \mathcal{H}_n$ has property $\mathcal{P}_4$. By definition, its complement $G^c$ has property $\mathcal{P}_3$. Applying Lemma \ref{lem:l0} to $G^c \in \mathcal{H}_n$, we obtain that $G^c \cup K_1$ has property $\mathcal{P}_4$. Since $\mathcal{P}_3$ and $\mathcal{P}_4$ are dual properties under graph complementation, and $(G^c \cup K_1)^c = G \vee K_1$, it follows that $G \vee K_1$ has property $\mathcal{P}_3$.
	\end{proof}
	
	Using Lemma \ref{lem:l0} and Lemma \ref{lem:l0_dual}, we can now prove the main theorem.
	\begin{proof}[Proof of Theorem \ref{main}]
		By assumption, the initial graph $G_0 \in \mathcal{H}_{n_0}$ satisfies the conditions of Theorem \ref{lnc} and Eq.~\eqref{mg}, which is precisely the statement that $G_0$ possesses property $\mathcal{P}_3$. 
		
		By the recursive definition of $G_i$, we can apply Lemma \ref{lem:l0} and Lemma \ref{lem:l0_dual} alternately. A straightforward induction shows that $G_i$ possesses property $\mathcal{P}_3$ when $i$ is even, and property $\mathcal{P}_4$ when $i$ is odd. 
		
		Since both $\mathcal{P}_3$ and $\mathcal{P}_4$ imply property $\mathcal{P}_2$, every graph $G_i$ in the sequence is almost controllable and satisfies the conditions of Theorem \ref{lnc2}. Consequently, each $G_i$ is uniquely determined by its generalized spectrum. This completes the proof of Theorem \ref{main}.
	\end{proof}
	\section*{Acknowledgments}
	This work is partially supported by the National Natural Science Foundation of China (Grant No. 12001006) and Wuhu Science and Technology Project, China (Grant No.~2024kj015). The authors acknowledge the use of Gemini 3.1 Pro for language polishing and detailed proofreading of the manuscript.

\end{document}